\documentclass[12pt,amstex]{amsart}

\usepackage{mathptmx}%字体
\usepackage{mathrsfs}

\usepackage{verbatim}
\usepackage{url}
\usepackage[all]{xy}
\usepackage{color}

\usepackage[colorlinks=true,citecolor=blue]{hyperref}%引用和定理颜色

\usepackage{stmaryrd}
\usepackage{epsfig}
\usepackage{amsmath}
\usepackage{amssymb}

\usepackage{amscd}
\usepackage{graphicx}
\usepackage{pstricks}
\usepackage{tocvsec2}

\newtheorem{conj}{Conjecture}

\newtheorem{thm}{Theorem}[section]
\newtheorem{lem}[thm]{Lemma}

\newtheorem{prop}[thm]{Proposition}
\newtheorem{cor}[thm]{Corollary}

\theoremstyle{definition}

\theoremstyle{remark}
\newtheorem{rem}[thm]{Remark}

\numberwithin{equation}{section}

\def \N {\mathbb N}

\def \Z {\mathbb Z}
\def \R {\mathbb R}

\def \T {\mathbb{T}}

\def \B {\mathcal{B}}

\def \X {\mathcal{X}}
\def \Y {\mathcal{Y}}
\def \ZZ {\mathcal{Z}}

\def \a {\alpha }
\def \b {\beta}
\def \ep {\epsilon}
\def \d {\delta}

\def\w {\omega}

\begin{document}
\title[A counterexample on polynomial multiple convergence]{A counterexample on polynomial multiple convergence without commutativity}% on systems with zero entropy}

\author{Wen Huang}
\author{Song Shao}
\author{Xiangdong Ye}

\address{CAS Wu Wen-Tsun Key Laboratory of Mathematics, and School of Mathematical Sciences, University of Science and Technology of China, Hefei, Anhui, 230026, P.R. China}

\email{wenh@mail.ustc.edu.cn}
\email{songshao@ustc.edu.cn}
\email{yexd@ustc.edu.cn}



\thanks{This research is supported by National Natural Science Foundation of China (12031019, 12090012, 12371196, 11971455).}

%\date{Feb. 6, 2023}

\begin{abstract}
It is shown that for polynomials $p_1, p_2\in {\mathbb Z}[t]$ with $\deg p_1, \deg p_2\ge 5$ there exist a probability space $(X,{\mathcal X},\mu)$, two ergodic measure preserving transformations $T,S$ acting on $(X,{\mathcal X}, \mu)$ with $h_\mu(X,T)=h_\mu(X,S)=0$, and $f, g \in L^\infty(X,\mu)$ such that the limit
\begin{equation*}
  \lim_{N\to\infty}\frac{1}{N}\sum_{n=0}^{N-1} f(T^{p_1(n)}x)g(S^{p_2(n)}x)
\end{equation*}
does not exist in $L^2(X,\mu)$, which in some sense answers a problem by Frantzikinakis and Host.

% in \cite{FranHost21}.

%In this paper, we show that for polynomial $p_1,p_2\in \Z[t]$ with $\deg p_1,\deg p_2\ge 5$ there exist a probability space $(X,\X,\mu)$ and two measure preserving %transformations $T,S$ acting on $(X,\X,\mu)$ with $h_\mu(X,T)=h_\mu(X,S)=0$ and $f,g \in L^\infty(X,\mu)$ such that the limit
%\begin{equation*}
%  \lim_{N\to\infty}\frac{1}{N}\sum_{n=0}^{N-1} f(T^{p_1(n)}x)g(S^{p_2(n)}x)
%\end{equation*}
%does not exist in $L^2(X,\mu)$, which partially answers a question by Frantzikinakis and Host in \cite{FranHost21}.
\end{abstract}

\maketitle

%\markboth{ergodic}{S. Shao and X.D. Ye}

%\newpage

%\tableofcontents \settocdepth{section}

%\newpage

\section{Introduction}

In the article, the set of integers (resp. natural numbers $\{1,2,\ldots\}$) is denoted by $\Z$ (resp. $\N$). By a {\em measure-preserving system} (m.p.s. for short), we mean a quadruple $(X,\X, \mu, T)$, where $(X,\X,\mu )$ is a Lebesgue probability space and $T: X \rightarrow X$ is an invertible measure preserving transformation.

Let $T,S$ be measure preserving transformations acting on a Lebesgue probability space $(X,\X, \mu)$. Then
for $f,g\in L^\infty(X,\mu)$, the existence in $L^2(X,\mu)$ of the limit
\begin{equation}\label{aa}
  \lim_{N\to\infty}\frac{1}{N}\sum_{n=0}^{N-1} f(T^nx)g(S^nx)
\end{equation}
was established in the commutative case by Conze and Lesigne \cite{CL84}, and the extension to
any finite number of transformations spanning a nilpotent group and to polynomial iterates
was established by Walsh \cite{Walsh}. When $T,S$ do not generate a nilpotent group, then \eqref{aa} may not exist \cite{BL02}.
Moreover, if both $T$ and $S$ have positive entropy, \eqref{aa} may fail to exist, see for example \cite[Proposition 1.4]{FranHost21}.

Recently,  Frantzikinakis and Host  \cite{FranHost21} gave the following beautiful multiple convergence theorem without commutativity.

\medskip

\noindent {\bf Theorem F-H.}
{\em
Let $T,S$ be measure preserving transformations acting on a probability
space $(X,\X,\mu)$ such that the system $(X,\X,\mu,T)$ has zero entropy. Let also $p\in \Z[t]$ be a
polynomial with $\deg (p) \ge  2$. Then for every $f,g\in L^\infty(X,\mu)$, the limit
\begin{equation}\label{bb}
  \lim_{N\to\infty}\frac{1}{N}\sum_{n=0}^{N-1} f(T^nx)g(S^{p(n)}x)
\end{equation}
exists in $L^2(X,\mu)$.
}

\medskip

It is unknown that whether {Theorem F-H} holds  when one replaces the iterates $n,p(n)$ by the pair $n,n$ or $n^2,n^3$ or
arbitrary polynomials $p,q\in \Z[t]$ with $p(0) = q(0) = 0$ (except the case covered by the above theorem), see the problem in \cite[Page 4]{FranHost21}.
In this paper we obtain a partial answer to this problem.

%In this paper, in some sense we obtain an answer to this problem.

%, i.e. we show that if the degrees of polynomials involved are not less than 3,
%then Theorem  does not hold.

\medskip

To be precise, the main result of the paper is as follows.

\medskip

\noindent {\bf Main Theorem.}
{\em Let $p_1, p_2:\Z\rightarrow \Z$ be polynomials with $\deg p_1,\deg p_2\ge 5$. For any $F\subset\N$ and $c \in (0,\frac12)$, there exist a Lebesgue probability space $(X,\X,\mu)$, two ergodic measure preserving transformations $T,S$ acting on $(X,\X,\mu)$ with $h_\mu(X,T)=h_\mu(X,S)=0$, and there are measurable subsets $A_1, A_2\in \X$ and $M\in \N$ such that for all $n\ge M$
\begin{equation}\label{abcd}
  \mu(A_1\cap T^{-p_1(n)}A_2\cap S^{-p_2(n)}A_2)=
\left\{
     \begin{array}{ll}
       0, & \hbox{if $n\in F$;} \\
       c, & \hbox{if $n\not \in  F$.}
     \end{array}
   \right.
\end{equation}

As a consequence, there exist a Lebesgue probability space $(X,\X,\mu)$, two ergodic measure preserving transformations $T,S$ acting on
$(X,\X,\mu)$ with $h_\mu(X,T)=h_\mu(X,S)=0$, and there is a measurable subset $A_2\in \X$ with $\mu(A_2)>0$ such that the averages
$$\frac{1}{N}\sum_{n=0}^{N-1}1_{A_2}(T^{p_1(n)}x)1_{A_2}(S^{p_2(n)}x)$$
do not converge in $L^2(X,\mu)$ as $N\to\infty$.
}

\medskip

\begin{rem} \label{remark11}We have the following remarks
\begin{enumerate}
\item By the proof of the Main Theorem, we may replace the condition that $p_1,$ $p_2$ $\in \Z[t]$ with degrees $\ge 5$
by any functions $h:\N\rightarrow \N$ satisfying for some $N\in \N$
\begin{equation*}
  \sum_{n=N}^\infty \sum_{k=1}^\infty \frac{1}{\sqrt{|h(n+k)-h(n)|}}<\infty \quad \text{and}\quad  \sum_{n=1}^\infty \frac{1}{\sqrt{h(n)}}<\infty.
\end{equation*}

%Also in the Main Theorem, the constant $c$ can be taken to be any number in $(0,\frac 12)$, and $S_0$ can be taken to be any infinite subset of $[M,+\infty)\setminus E$ with density $0$.

\item By (1) and Lemma \ref{easylemma}, particularly the Main Theorem applies to the question concerning the convergence in $L^2$ of the following
averages
$$\frac{1}N\sum_{n=1}^Nf(T^{[n^a]}x)g(S^{[n^b]}x),$$
when $a,b>4$ are reals, which was suggested in \cite{Nikos20} (see the sentences after \cite[Problem 2]{Nikos20}).
%For details of the study when $a=1$ and $b>0$ ($b\not\in\N$), see.

%\item Due to some technical reason of our proof, we need to introduce the subset $S_0$. See Remark \ref{rem-Re} for a way to remove $S_0$.

%We think that it can be removed by showing more properties of the function $f$ appearing in Theorem \ref{thm-KV}.

\item It seems that the systems we constructed in the Main Theorem are not measurable distal. It will be very nice if measurable distal systems can be constructed.
\end{enumerate}
\end{rem}

\begin{conj}
Theorem F-H is sharp in the sense that the pair $cn, p(n)$ with $\deg(p)\ge 2, c\in \Z\setminus\{0\}$ can not be replaced by any other pair $p_1,p_2$, where $p_1,p_2\in \Z[t]$
with $p_1(0)=p_2(0)=0$. % are integral polynomials.
\end{conj}

Now we briefly describe the ideas of the proof of the Main Theorem. First we fix a function $f$ in the so-called {\it the lattice local central limit theorem}
proved by Kosloff and Volny \cite[Theorem]{KosloffVolny22}, see Theorem \ref{thm-KV}. The probability space $(X,\X,\mu)$ appeared in the Main Theorem is
$(\T\times \Sigma, \B(\T\times \Sigma),\mu)$, where $\T$ is the unit circle, $\Sigma=\{0,1\}^\Z$, $\B(\T\times \Sigma)$ is the Borel $\sigma$-algebra
and $\mu$ is the product measure.

The transformation $T$ that we construct is a skew product of an irrational rotation on $\T$ with the full shift $\sigma$ on $\Sigma$ related to the given function $f:\T\rightarrow \Z$, i.e. $(y, \w)\mapsto (y+\a, \sigma^{f(y)}\w)$ for $(y,\w)\in \T\times \Sigma$.
Using the properties of $f$ stated in Theorem \ref{thm-KV} and the Abramov-Rokhlin formula, we are able to show that $T$ is ergodic with zero entropy.
%(the properties of $f$ and the assumptions $\deg(p_1),\deg(p_2)\ge 5$ are essentially used here).
%\medskip
Then we construct the transformation $S$. The construction of $S$ is very much involved, since at one hand to make the proof of the ergodicity
and zero entropy property of $S$ easier, we require that $S$ is isomorphic to $T$; and at the other hand we need to guarantee (\ref{abcd}) hold
(the subsets $A_1$ and $A_2$ appear naturally in the process of the construction).
We remark that the assumptions of $\deg (p_1), \deg(p_2)\ge 5$ are essentially used in the construction of $S$, see Lemma \ref{lem-1}. As all terms have been
considered in the constructions of $T$ and $S$, the Main Theorem follows.
%Note that due to some technical reason we need the subset $S_0$, and we think that it can be removed.

\medskip
\noindent{\bf Acknowledgments:} The authors thank Frantzikinakis for suggesting to add Remark \ref{remark11}(3), and Kosloff for
a correction concerning the degree of the polynomials we consider.
We also would like to thank the referee for the very careful reading and many useful comments.
For example, in the original version of the paper, we need an additional subset $S_0$ in the Main Theorem, and
the referee suggested us to study $2f$ instead of $f$ to remove this $S_0$. We thank the referee for pointing out this to us, which simplified statement of the main result.

\section{Proof of the main result}

In this section, we prove the Main Theorem. First we do some preparations, and then we construct $(X,\X,\mu)$ and $T, S$ on it with desired properties.

\subsection{The local central limit theorem}
The following theorem plays an important role in our construction.
\begin{thm}\cite{KosloffVolny22}\label{thm-KV}
For every ergodic, aperiodic and probability measure preserving system $(Y,\Y, m, R)$ there exists
a square integrable function $f: Y\rightarrow \Z$ with $\int_Y f dm=0$ which satisfies the lattice local central limit theorem. That is,  $\lim_{n\to\infty} \frac{\|S_n(f)\|_2^2}{n}=\sigma^2>0$ and
\begin{equation}\label{a0}
  \sup_{x\in \Z}\left|\sqrt{n}\cdot m\Big(S_n(f)=x\Big)-\frac{e^{-x^2/(2n\sigma^2)}}{\sqrt{2\pi\sigma^2}}\right| \overset{n\to\infty}\longrightarrow 0,
\end{equation}
where $S_n(f):= \sum_{k=0}^{n-1} f\circ R^k$.
\end{thm}

\subsection{Abramov-Rokhlin formula }
We will need Abramov-Rokhlin formula to calculate entropy. For the definition of entropy, refer to \cite{Glasner}.

Let $(Y,\Y,m,R)$ be a m.p.s. and $\{\phi(y): y\in Y\}$ be a family of measurable
transformations of a measurable space $(Z,\ZZ,\nu)$ such that the map $(y,z)\mapsto \phi(y)z$ is
measurable from $Y\times Z$ to $Z$. We can then define the {\em skew-product transformation} $T$ on
$Y\times Z$ by
$$T(y,z)=(Ry,\phi(y)z), \ \forall (y,z)\in Y\times Z.$$ Let $X=Y\times Z, \X=\Y\times \ZZ$ and $\mu=m\times \nu$. Then $(X,\X, \mu,T)$ is a m.p.s.
For any finite  measurable partition $\b$ of $Z$ the limit
$$h_\mu(T|R, \b)=\lim_{n\to\infty}\frac{1}{n} \int_Y H_\nu(\bigvee_{i=0}^{n-1}\phi^{-1}_i(y)\b) d m(y)$$
exists, where  $\phi_0(y)$ is the identity map on $Z$ and $\phi_i(y)=\phi(R^{i-1}y)\circ \cdots \circ \phi(y)$ for $i\in \mathbb{N}$, and
$$h_\mu(T|R)=\sup\{h_\mu(T|R, \b): \b \ \text{is a finite  measurable partition of}\ Z\}$$
is called the {\em fiber-entropy of $T$}.
Abramov-Rokhlin formula \cite{AR} says that
\begin{equation}\label{}
  h_\mu(T)=h_m(R)+h_\mu(T|R).
\end{equation}

\subsection{Additional assumptions on $p_1,p_2$} \label{sub-assumption}\
\medskip

To prove the Main Theorem, we may additionally assume that the leading coefficients of polynomials involved are positive.
That is, we may assume $p_1(n)=a_tn^t+a_{t-1}n^{t-1}+\cdots +a_1 n+a_0$ and $p_2(n)=b_sn^s+b_{s-1}n^{s-1}+\cdots +b_1n+b_0$
with $a_t>0$ and $b_s>0$, and $t,s\ge 5$.

\medskip

To see this, we assume that the Main Theorem  holds for polynomials with positive leading coefficients. Now we show the case when $p_1$ has a negative leading coefficient and $p_2$ has a positive leading coefficient. Since $-p_1$ and $p_2$ have positive leading coefficients, by the hypothesis there exist a  Lebesgue probability space $(X,\X,\mu)$, two ergodic measure preserving transformations $\widetilde{T},S$ acting on $(X,\X,\mu)$ with $h_\mu(X,\widetilde{T})=h_\mu(X,S)=0$ and there are measurable subsets $A_1,A_2\in \X$, $M\in \N$, a positive $c>0$
such that for all $n\ge M$
\begin{equation*}
  \mu(A_1\cap \widetilde{T}^{-(-p_1(n))}A_2\cap S^{-p_2(n)}A_2)=
\left\{
     \begin{array}{ll}
       0, & \hbox{if $n\in F $;} \\
       c, & \hbox{if $n\not \in F$.}
     \end{array}
   \right.
\end{equation*}
Let $T=\widetilde{T}^{-1}$. Then we still have $h_\mu(X,{T})=0$ and the same $A_1, A_2, M, c$ %there are measurable subsets $A_1,A_2\in \X$, $M\in \N$ and a positive $c>0$
such that for all $n\ge M$
\begin{equation*}
  \mu(A_1\cap {T}^{-p_1(n)}A_2\cap S^{-p_2(n)}A_2)=
\left\{
     \begin{array}{ll}
       0, & \hbox{if $n\in F $;} \\
       c, & \hbox{if $n\not \in F$.}
     \end{array}
   \right.
\end{equation*}
That is, we still have the Main Theorem in this case. Similarly, we deal with the remaining cases.
%case when  $p_2$ has a negative leading coefficient and $p_1$ has a positive leading coefficient, and the case when both $p_1$ and $p_1$ have %negative leading coefficients.

\subsection{Construction of the m.p.s. $(X,\X,\mu,T)$}

\subsubsection{} For a topological space $X$, we use $\B(X)$ denote the $\sigma$-algebra generated by all open subsets of $X$.

Let $\T=\R/\Z$ be the circle.
Let $(\T, \B(\T),m, R_\a)$ be the circle rotation by an irrational number $\a$, where $m$ is the Lebesgue measure on $\T$ and
$$R_\a: \T\rightarrow \T,\  y\mapsto y+\a \pmod{1}.$$
By Theorem \ref{thm-KV}, there exists
a square integrable  Borel function $f: \T \rightarrow \Z$ with $\displaystyle \int_\T f dm=0$ which satisfies the lattice local central limit theorem \eqref{a0}. Let $g=2f$.
{\bf We fix such Borel functions  $f$ and $g$ in the sequel}.

\medskip

Let $\Sigma=\{0,1\}^\Z$, and $\big(\Sigma, \B(\Sigma), \nu ,\sigma\big)$ be the $(\frac12,\frac 12)$-shift. That is, $\nu=(\frac 12\d_0+\frac12\d_1)^\Z$ is the product measure on $\Sigma=\{0,1\}^\Z$ and for each sequence $\w\in \Sigma$,
$$(\sigma \w)(n)=\w(n+1), \quad \forall n\in \Z.$$

Now let $(X,\X,\mu)=(\T\times \Sigma, \B(\T)\times \B(\Sigma),m\times \nu)$, and define
\begin{equation}\label{}
  T: \T\times \Sigma\rightarrow \T\times \Sigma, \quad (y, \w)\mapsto (y+\a, \sigma^{g(y)}\w).
\end{equation}
Since $g=2f: \T\rightarrow \Z$ is measurable, it is easy to verify that $(X,\X,\mu,T)$ is a m.p.s.

\subsubsection{}

Let  $f_n=S_n(f)=\sum_{k=0}^{n-1} f\circ R_\a^k$ and $g_n=S_n(g)=\sum_{k=0}^{n-1} g \circ R_\a^k=2f_n$ for $n\in \mathbb{N}$. Then for all $n\ge 1$ and $(y,\w)\in X$, we have
\begin{equation}\label{b1}
  T^n(y,\w)=(R_\a^n(y), \sigma^{g_n(y)}\w)=(y+n\a, \sigma^{g_n(y)}\w).
\end{equation}

\subsubsection{}

For $a, b\in \mathbb{Z}$ with $a\le b$ and $(s_1,s_2,\ldots, s_{b-a+1})\in \{0,1\}^{b-a+1}$, let
\begin{equation}\label{}
  _{a}[s_1,s_2,\ldots, s_{b-a+1}]_b:=\{\w\in \Sigma: \w(a)=s_1,\ldots,\w(b)=s_{b-a+1}\}.
\end{equation}
And for $i\in \{0,1\}$, $j\in \Z$ let
$$[i]_j:= {_j[i]_j}=\{\w\in \Sigma: \w(j)=i\}.$$

\subsubsection{}
Now we show $(X,\X,\mu,T)$ is ergodic and its entropy is zero.

\begin{prop}\label{prop-zero}
$(X,\X,\mu,T)$ is an ergodic m.p.s. with $h_\mu(X,T)=0$.
\end{prop}

\begin{proof}
%Since $f: \T\rightarrow \Z$ is measurable, it is easy to verify that $(X,\X,\mu,T)$ is a m.p.s.
First we show that $(X,\X,\mu,T)$ is ergodic. Let $A_1,A_2\in \B(\T)$ with $m(A_1)m(A_2)>0$ and $B_j={_{-M_j}}[s^j]_{M_j}\in \B(\Sigma)$, where $s^j\in \{0,1\}^{2M_j+1}$
with some $M_j\in \N$, $j=1,2$. For each $n\in\N$ set
$$W_n=\{y\in \T: |g_n(y)|\le M_1+M_2\}.$$
By \eqref{b1},
\begin{equation*}
  \begin{split}
     & \mu\big((A_1\times B_1)\cap T^{-n}(A_2\times B_2)\big) \\
     = & \int_\Sigma \int_\T 1_{R_\a^{-n}A_2\cap A_1}(y)\cdot 1_{\sigma^{-g_n(y)}B_2\cap B_1}(\w) dm(y)d\nu(\w).
  \end{split}
\end{equation*}
Note that when $|q|>M_1+M_2$, $\nu(\sigma^{-q}B_1\cap B_2)=\nu(B_1)\nu(B_2)$. Thus
\begin{equation}\label{a5}
  \begin{split}
     & \mu\big((A_1\times B_1)\cap T^{-n}(A_2\times B_2)\big) \\
     = & \int_\Sigma \int_\T 1_{R_\a^{-n}A_2\cap A_1}(y)\cdot 1_{\sigma^{-g_n(y)}B_2\cap B_1}(\w) dm(y)d\nu (\w)\\
\ge & \int_\Sigma \int_{\T\setminus W_n} 1_{R_\a^{-n}A_2\cap A_1}(y)\cdot 1_{\sigma^{-g_n(y)}B_2\cap B_1}(\w) dm(y)d\nu (\w)\\
= & \int_{\T\setminus W_n} 1_{R_\a^{-n}A_2\cap A_1}(y)\cdot \nu(B_1)\nu(B_2) dm(y)\\
\ge & \Big(m(R_\a^{-n} A_2\cap A_1)-m(W_n)\Big) \nu(B_1)\nu (B_2).
  \end{split}
\end{equation}
Now we estimate $m(W_n)$. For $j\in \Z$, by \eqref{a0}, there is some $N_j\in \N$ such that for $n>N_j$
$$\left|\sqrt{n}\cdot m\Big(f_n=j\Big)-\frac{e^{-j^2/(2n\sigma^2)}}{\sqrt{2\pi\sigma^2}}\right|< \frac{1}{\sqrt{2\pi\sigma^2}}. $$
Thus when $n>N_j$ we have
$$m\Big(f_n=j\Big)< \frac{e^{-j^2/(2n\sigma^2)}}{\sqrt{n}\sqrt{2\pi\sigma^2}}+\frac{1}{\sqrt{n}\sqrt{2\pi\sigma^2}}\le \frac{2}{\sqrt{n}\sqrt{2\pi\sigma^2}}.$$
Since $g_n=2f_n$, we have that for all $j\in \Z$,
$$m\Big(g_n=2j\Big)=m\Big(f_n=j\Big)< \frac{2}{\sqrt{n}\sqrt{2\pi\sigma^2}}, \ \text{and} \ m\Big(g_n=2j+1\Big)=0.$$
Hence for $n> \max_{|j|\le M_1+M_2}\{N_j\}$
\begin{equation*}
  \begin{split}
  m(W_n) & =\sum_{j=-(M_1+M_2)}^{M_1+M_2}m(g_n=j)= \sum_{j=-\left[\frac{M_1+M_2}{2}\right]}^{\left[\frac{M_1+M_2}{2}\right]} m(f_n=j) \\
 & <  \sum_{j=-\left[\frac{M_1+M_2}{2}\right]}^{\left[\frac{M_1+M_2}{2}\right]} \frac{2}{\sqrt{n}\sqrt{2\pi\sigma^2}} =\frac{2(2\left[\frac{M_1+M_2}{2}\right]+1)}{\sqrt{n}\sqrt{2\pi\sigma^2}}.
  \end{split}
\end{equation*}
It follows that $\lim_{n\to\infty}m(W_n)=0$.

By \eqref{a5} and the ergodicity of $(\T,\B(\T), m,R_\a)$, we deduce
\begin{equation*}
  \begin{split}
     & \lim_{N\to\infty}\frac{1}{N}\sum_{n=0}^{N-1}\mu\big((A_1\times B_1)\cap T^{-n}(A_2\times B_2)\big) \\ \ge & \lim_{N\to\infty}\frac{1}{N}\sum_{n=0}^{N-1} \Big(m(R^{-n}_\a A_2\cap A_1)-m(W_n)\Big) \nu (B_1)\nu (B_2)\\
= & m(A_1)m(A_2) \nu (B_1)\nu (B_2)=\mu(A_1\times B_1)\mu(A_2\times B_2).
  \end{split}
\end{equation*}
Then it is standard to prove that for all $D_1,D_2\in \X$, we have that
$$\lim_{N\to\infty}\frac{1}{N}\sum_{n=0}^{N-1}\mu\big(D_1\cap T^{-n}D_2\big)\ge \mu(D_1)\mu(D_2). $$
In particular, we have that for any $D_1,D_2\in \X$ with $\mu(D_1)\mu(D_2)>0$, there is some $n\in \N$ such that $\mu\big(D_1\cap T^{-n}D_2\big)>0$, which means that $(X,\X,\mu,T)$ is ergodic.

\medskip

Now we use Abramov-Rokhlin formula to show that $h_\mu(T)=0$. For any finite measurable partition $\b$ of $\Sigma$, we have
\begin{equation*}
  h_\mu(T|R_\a, \b)=\lim_{n\to\infty}\frac{1}{N} \int_\T H_\nu\Big(\bigvee_{n=0}^{N-1}\sigma^{-g_n(y)}\b\Big) d m(y),
\end{equation*}
where $g_0(y)\equiv 0$.

For $y\in \mathbb{T}$ and $N\in\N$, we denote $a_N(y)$ the cardinality of the set $\{g_n(y): 0\le n\le N-1\}$, i.e.
$$a_N(y)=\left|\{g_n(y): 0\le n\le N-1\}\right|.$$
%where $|A | $ means the cardinality of the set $A$.
Then $a_N$ is a measurable function from $\mathbb{T}$ to $\{1,2,\cdots,N\}$, and the cardinality of $\bigvee_{n=0}^{N-1}\sigma^{-g_n(y)}\b$ is not greater than $|\b|^{a_N(y)}$ for any $y\in \mathbb{T}$, and hence
$$\frac{1}{N} \int_\T H_\nu(\bigvee_{n=0}^{N-1}\sigma^{-g_n(y)}\b) d m(y)\le \frac{1}{N} \int_\T \log |\b|^{a_N(y)}d m(y)= \int_\T \frac{a_N(y)}{N}\log |\b| d m(y).$$
We claim that
for $m$-a.e. $y\in \T$,
\begin{equation}\label{c3}
  \lim_{N\to\infty} \frac{a_N(y)}{N}=0.
\end{equation}
We now show the claim. Since $(\T,\B(\T),m,R_\a)$ is ergodic and $\int_\T f dm=0$, by Birkhoff ergodic theorem, for $m$-a.e. $y\in \T$,
$$\frac{f_n(y)}{n}=\frac{1}{n}\sum_{k=0}^{n-1}f(R^k_\a y)\to \int_\T fd m=0, \ n\to\infty.$$
Thus for $m$-a.e. $y\in \T$, for any $\ep>0$ there is some $M(y,\ep)\in \N$ such that when $ n\ge M(y,\ep)$, we have
$\left|\frac{g_n(y)}{n}\right|=\left|\frac{2f_n(y)}{n}\right|\le \ep$. Thus for all $N\ge n\ge M(y,\ep)$, we have $|g_n(y)|\le \ep n\le \ep N$. It follows that
$a_N(y)\le M(y,\ep)+2\ep N+1,$
and
$$ \lim_{N\to\infty} \frac{a_N(y)}{N}\le  \lim_{N\to\infty} \frac{M(y,\ep)+2\ep N+1}{N}\le 2\ep.$$
Since $\ep$ is arbitrary, we have \eqref{c3}, i.e. for $m$-a.e. $y\in \T$,
$\lim \limits_{N\to\infty} \frac{a_N(y)}{N}=0.$ This ends the proof of the claim.

Thus by Dominated Convergence Theorem,
\begin{align*}
h_\mu(T|R_\a, \b)&=\lim_{N\to\infty}\frac{1}{N} \int_\T H_\nu\Big(\bigvee_{n=0}^{N-1}\sigma^{-g_n(y)}\b\Big) d m(y)\\
&\le \log |\b| \lim_{N\to\infty}\int_\T \frac{a_N(y)}{N}d m(y)=\log |\b|\int_\T \lim_{N\to\infty}\frac{a_N(y)}{N}d m(y)\\
&=0.
\end{align*}
As $\b$ is an arbitrary finite measurable partition of $\Sigma$, we have $h_\mu(T|R_\a)=0$. Then by Abramov-Rokhlin formula,
$$h_\mu(T)=h_m(R_\a)+h_\mu(T|R_\a)=0.$$
The proof is complete.
\end{proof}

\subsection{Construction of the m.p.s. $(X,\X,\mu,S)$}

We now construct the transformation $S$. As we said in the introduction, this construction is much involved.

\subsubsection{}
First we need some lemmas. By Subsection \ref{sub-assumption}, we assume that all polynomials considered have positive leading coefficients.

\begin{lem}\label{easylemma}
Let $h(n)=p(n)$ be a polynomial with a positive leading coefficient and $\deg p\ge 5$ or
$h(n)=[n^a]$ with $a>4$, where $[\cdot]$ is the integral part of a real number. Then there is some $M_1\in \N$ such that
\begin{equation*}\label{degree5}
  \sum_{n=M_1}^\infty \sum_{k=1}^\infty \frac{1}{\sqrt{h(n+k)-h(n)}}<\infty.
\end{equation*}
\end{lem}

\begin{proof}
First let $p(n)=\sum_{j=0}^t a_jn^j$ with $a_t>0, t\ge 5$. Let $q(n)=p(n+1)-p(n)$. Then
$$\lim_{n\to\infty} \frac{q(n)}{(n+1)^t-n^t}=\lim_{n\to\infty} \frac{p(n+1)-p(n)}{(n+1)^t-n^t}=a_t>0.$$
Thus there is some $M_1\in \N$ such that for all $n\ge M_1$ we have
$$q(n)\ge \frac{a_t}{2}\big( (n+1)^t-n^t \big)>0.$$
Thus for all $n\ge M_1$ and $k\in \N$,
\begin{equation}\label{aaa-1}
  \begin{split}
    & p(n+k)-p(n)=\sum_{i=0}^{k-1}q(n+i)\ge \sum_{i=0}^{k-1}\frac{a_t}{2}\big( (n+i+1)^t-(n+i)^t \big)\\ =& \frac{a_t}{2}\big( (n+k)^t-n^t \big)\ge \frac{a_t}{2}\big( tn^{t-1}k+t nk^{t-1}\big)\\
      \ge & a_t t n^{ t/ 2}k^{ t/ 2}.
   \end{split}
\end{equation}
It follows that
\begin{equation*}
  \begin{split}
    & \sum_{n=M_1}^\infty \sum_{k=1}^\infty \frac{1}{\sqrt{p(n+k)-p(n)}}  \le \sum_{n=M_1}^\infty \sum_{k=1}^\infty \frac{1}{\sqrt{a_t t n^{ t/ 2}k^{ t/ 2}}}\\
     = & \frac{1}{\sqrt{a_t t}}\sum_{n=M_1}^\infty \sum_{k=1}^\infty \frac{1}{ n^{ t/ 4}k^{ t/ 4}}\le \frac{1}{\sqrt{a_t t}}\Big(\sum_{n=M_1}^\infty \frac{1}{ n^{ t/ 4}}\Big)\Big(\sum_{k=1}^\infty \frac{1}{ k^{ t/ 4}}\Big).
   \end{split}
\end{equation*}
Since $t\ge 5$, we have $\displaystyle \sum_{n=M_1}^\infty \sum_{k=1}^\infty \frac{1}{\sqrt{p(n+k)-p(n)}} <\infty$.

\medskip
Now assume that $h(n)=[n^a]$ with $a>4$. We note that for $\ep=a-4>0$, $n,k\in\N$, we have
\begin{equation}\label{aaa-2}
(n+k)^\ep=(n^2+2nk+k^2)^{\ep/2}>c(nk)^{\ep/2},
\end{equation} where $c>0$ is a constant. Moreover,
\begin{equation}\label{aaa-3}
[(n+k)^a]-[n^a]\ge (n+k)^a-n^a-1\ge (n+k)^{a-4}((n+k)^4-n^4)-1.
\end{equation}
Thus, (\ref{aaa-1}), (\ref{aaa-2}) and (\ref{aaa-3}) give us the required property, and the proof is complete.
\end{proof}

\begin{rem}

\begin{enumerate}
  \item Based on Lemma \ref{easylemma}, it is natural to ask if we can generalize the lemma to the following form: if $h: \N_{\ge M}=\{n\in \N: n\ge M\}\rightarrow \N$, ($M\in \N$) is an increasing function satisfying $$\lim_{n\to\infty}\frac{h(n)}{n^{4+\ep}}=\infty$$ for some $\ep>0$, then we have
  $\displaystyle \sum_{n=M}^\infty \sum_{k=1}^\infty \frac{1}{\sqrt{h(n+k)-h(n)}} <\infty$.

The following example shows that this question has a negative answer.

%But this is not true. We have the following counterexample.
  Let $\displaystyle h(n)=\left[\frac n 2\right]^5+(-1)^{n+1}$. Then $h: \N_{\ge 3}\rightarrow \N$ is an increasing function, and for each $\ep\in (0,1)$, we have $\displaystyle\lim_{n\to\infty}\frac{h(n)}{n^{4+\ep}}=\infty$.
  Note that $$h(2m+1)-h(2m)=2, \ \forall m\in \N.$$
  Thus
  $$\sum_{n=3}^\infty \sum_{k=1}^\infty \frac{1}{\sqrt{h(n+k)-h(n)}} \ge \sum_{m=1}^\infty \frac{1}{\sqrt{h(2m+1)-h(2m)}}=\infty.$$

  \item
Now we consider the function $h(n)=[n^4\log^s n],\ s> 0$, and show that if $s>2$, then there is some $M\in \N$ such that
%\begin{equation*}
$  \sum_{n=M}^\infty \sum_{k=1}^\infty \frac{1}{\sqrt{h(n+k)-h(n)}}<\infty.$
%\end{equation*}
%But we do not know the case when $s\in (0,2]$.

To see this, note that
\begin{equation*}
  \begin{split}
    h(n+k)-h(n)&\ge (n+k)^4\log^s(n+k )-n^4\log^s n-2\\
     &\ge  \left((n+k)^4-n^4\right)\log^s(n+k)-2\\
    &=\left((n+k)^3n+(n+k)^3k-n^4\right)\log^s(n+k)-2\\
    &\ge  (n+k)^3k\log^s(n+k)-2.
   \end{split}
\end{equation*}
There is some $M\in \N$ such that whenever $n\ge M$, we have $$h(n+k)-h(n) \ge  (n+k)^3k\log^s(n+k)-2 \ge \frac{1}{2} (n+k)^3k\log^s(n+k).$$
Thus
\begin{equation*}
  \begin{split}
    \sum_{n=M}^\infty \sum_{k=1}^\infty \frac{1}{\sqrt{h(n+k)-h(n)}}  &\le  \sum_{k=1}^\infty \sum_{n=M}^\infty \frac{\sqrt{2}}{\sqrt{(n+k)^3k\log^s(n+k)}}\\
    &\le  \sum_{k=1}^\infty \left(\frac{\sqrt{2}}{k^{1/2}\log^{s/2}(M+k)}\sum_{n=M}^\infty \frac{1}{(n+k)^{3/2}}\right)\\
    &\le \sum_{k=1}^\infty \left(\frac{\sqrt{2}}{k^{1/2}\log^{s/2}(M+k)}\int_{n=M+k}^\infty \frac{1}{x^{3/2}}dx\right)\\
     &= \sum_{k=1}^\infty \left(\frac{2\sqrt{2}}{k^{1/2}\log^{s/2}(M+k)}\frac{1}{(M+k)^{1/2}}\right) \\
    &\le  2\sqrt{2} \sum_{k=1}^\infty \frac{1}{k\log^{s/2}(M+k)} <\infty \quad \quad ({\rm since} \ s>2).
   \end{split}
\end{equation*}
%Since $s>2$, we have that $\displaystyle \sum_{n=M}^\infty \sum_{k=1}^\infty \frac{1}{\sqrt{h(n+k)-h(n)}}  \le 2\sqrt{2}\sum_{k=1}^\infty \frac{1}{k\log^{s/2}(M+k)}<\infty.$

We do not know whether $\displaystyle \sum_{n=M}^\infty \sum_{k=1}^\infty \frac{1}{\sqrt{h(n+k)-h(n)}} <\infty$, if $0<s\le 2$.

\end{enumerate}

\end{rem}

\begin{lem}\label{lem-1}
Let $p:\Z\rightarrow \Z$ be a polynomial with a positive leading coefficient and $\deg p\ge 5$. For $N\in\N$ set
$$E_N(p)=\{ y\in \T: f_{p(n)}(y)\neq 0, f_{p(n)}(y)\neq f_{p(n+k)}(y)\ \text{for all }\ n\ge N, k\ge 1\}.$$
Then
\begin{equation*}\label{}
  \lim_{N\to\infty}m(E_N(p))=1.
\end{equation*}
\end{lem}

\begin{proof}
Since $p$ has a positive leading coefficient, there is some $N_1:=N_1(p)\in \mathbb{N}$ with $N_1\ge M_1$ ($M_1$ is defined in Lemma \ref{easylemma}) such that $p(n)$ is strictly monotone increasing on $[N_1,+\infty)$ and for $n\ge N_1$, $p(n)>0$. Thus for $n\ge N_1$, $f_{p(n)}$ is well-defined.

For $n\in \mathbb{N}$ with $n\ge N_1$, let $$G_n^0=\{y\in \T: f_n(y)=0\}.$$
For $n,k\in \N$ and $n\ge N_1$,we have
\begin{equation}\label{a1}
  \begin{split}
     F_{n,k} & = \{y\in \T: f_{p(n)}(y)=f_{p(n+k)}(y)\}\\
       & = \{y\in \T: f_{p(n+k)-p(n)}(y+p(n)\a)=0\}\\
       &=G^0_{p(n+k)-p(n)}-p(n)\a.
   \end{split}
\end{equation}

Then for $N\ge N_1$
\begin{equation}\label{a2}
  E_N(p)=\T\setminus \Big(\bigcup_{n=N}^\infty \bigcup_{k=1}^\infty F_{n,k} \cup\bigcup_{n=N}^\infty G_{p(n)}^0\Big).
\end{equation}
By \eqref{a1},
\begin{equation}\label{a3}
  m\Big(\bigcup_{n=N}^\infty \bigcup_{k=1}^\infty F_{n,k}\Big)\le \sum_{n=N}^\infty \sum_{k=1}^\infty m(F_{n,k})=\sum_{n=N}^\infty \sum_{k=1}^\infty m(G_{p(n+k)-p(n)}^0).
\end{equation}
By \eqref{a0},
\begin{equation*}
  \sup_{x\in \Z}\left|\sqrt{n}\cdot m\Big(f_n=x\Big)-\frac{e^{-x^2/(2n\sigma^2)}}{\sqrt{2\pi\sigma^2}}\right| \overset{n\to\infty}\longrightarrow 0,
\end{equation*}
where $\sigma>0$. In particular, for $x=0$ we have
\begin{equation*}
  \left|\sqrt{n}\cdot m\Big(G_n^0\Big)-\frac{1}{\sqrt{2\pi\sigma^2}}\right| \overset{n\to\infty}\longrightarrow 0.
\end{equation*}
Thus there is some $N_2:=N_2(p)\in \mathbb{N}$ such that for all $n\ge N_2$,
\begin{equation}\label{a4}
  m(G_n^0)\le \frac{2}{\sqrt{n}\sqrt{2\pi\sigma^2}}:=\frac{c}{\sqrt{n}}, \quad \text{where}\ c=\frac{2}{\sqrt{2\pi\sigma^2}}>0.
\end{equation}
Combining \eqref{a3} with \eqref{a4}, we conclude that for all $N\ge \max\{N_1,N_2\}$
\begin{equation*}
  m\Big(\bigcup_{n=N}^\infty \bigcup_{k=1}^\infty F_{n,k}\Big)\le \sum_{n=N}^\infty \sum_{k=1}^\infty m(G_{p(n+k)-p(n)}^0)\le \sum_{n=N}^\infty \sum_{k=1}^\infty \frac{c}{\sqrt{p(n+k)-p(n)}}
\end{equation*}
and
\begin{equation*}
  m\Big(\bigcup_{n=N}^\infty G_{p(n)}^0\Big)\le c \sum_{n=N}^\infty \frac{1}{\sqrt{p(n)}}.
\end{equation*}
Since $\deg p\ge 5$ and $N_1\ge M_1$, we have $\sum_{n=N_1}^\infty \sum_{k=1}^\infty \frac{1}{\sqrt{p(n+k)-p(n)}}<\infty$ and $\sum_{n=N_1}^\infty \frac{1}{\sqrt{p(n)}}<\infty$. Hence
$$\lim_{N\to\infty}m\Big(\bigcup_{n=N}^\infty \bigcup_{k=1}^\infty F_{n,k}\Big)=0\ \  \text{and}\quad \lim_{N\to\infty}m\Big(\bigcup_{n=N}^\infty G_{p(n)}^0\Big)=0. $$
Thus by \eqref{a2}, we derive that $\lim_{N\to\infty}m(E_N(p))=1.$ The proof is complete.
\end{proof}

Since $g_n=2f_n$, we have that 
$$E_N(p)=\{ y\in \T: g_{p(n)}(y)\neq 0, g_{p(n)}(y)\neq g_{p(n+k)}(y)\ \text{for all }\ n\ge N, k\ge 1\}.$$
An immediate consequence is %By Lemma \ref{lem-1}, we have:

\begin{cor}\label{cor-1}
Let $p_1,p_2:\Z\rightarrow \Z$ with positive leading coefficients and $\deg p_1,\deg p_2\ge 5$. Then for any $\eta\in (0,1)$ there exist a measurable subset $B\subseteq \T$
with $m(B)=\eta$, and $M\in \N$ such that for all $y\in B$,
all terms in $\{g_{p_j(n)}(y)\}_{n=M}^\infty$ are distinct and non-zero, $j=1,2$.
\end{cor}

\begin{proof}
By Lemma \ref{lem-1}, $\lim_{N\to\infty}m(E_N(p_1))=1$ and $\lim_{N\to\infty}m(E_N(p_2))=1$.
Thus  $$\lim_{N\to\infty}m\big(E_N(p_1)\cap E_N(p_2)\big)=1.$$
In particular, for any $\eta\in (0,1)$ there is some natural number $M>\max\{N_i(p_j),N_i(p_j):i,j=1,2\}$ ($N_1,N_2$ are defined in the proof of Lemma \ref{lem-1}) such that
$$m\big(E_M(p_1)\cap E_M(p_2)\big)\ge \eta.$$ 
Choose a measurable subset $B\subseteq E_M(p_1)\cap E_M(p_2)$ such that $m(B)=\eta$. Since $B\subseteq E_M(p_1)\cap E_M(p_2)$, for all $y\in B$,
all terms in $\{g_{p_j(n)}(y)\}_{n=M}^\infty$ are distinct and non-zero, $j=1,2$. 
\end{proof}

\subsubsection{Construction of $(X,\X,\mu,S)$}
Let $p_1,p_2, B, M$ be as defined in Corollary \ref{cor-1}. For each $y\in B$ we define a permutation $\pi_y$ of $\Z$ such that
$$0\mapsto 0, \quad  g_{p_2(n)}(y)\mapsto g_{p_1(n)}(y),\quad  \forall n\in [M,+\infty).$$
Our goal is to construct a $\pi_y$ that is a
measurable function of $y$. We will do so as follows.

\medskip
Note that $\big\{g_{p_1(n)}(y):n\in [M,+\infty)\big\}=\{g_{p_1(n)}(y)\}_{n=M}^\infty \subseteq 2\Z$, and hence $\Z\setminus \{g_{p_1(n)}(y)\}_{n=M}^\infty$ is also infinite.
Now we enumerate $\Z\setminus \{0\}$ with $\{l_i\}_{i=1}^\infty$. For example, we
put $l_i:=(-1)^{i-1}\left[\frac{i+1}{2}\right],\ i\ge 1$.
%put $$l_i:=(-1)^{i-1}\left[\frac{i+1}{2}\right],\ i\ge 1,\ \text{i.e.}\ \{l_i\}_{i=1}^\infty=\{1,-1,2,-2,\ldots\}=\Z\setminus \{0\}.$$
Let
$$L(y)=\{j\ge 1: l_j\not\in \{g_{p_1(n)}(y)\}_{n=M}^\infty \}:=\{j_1(y)<j_2(y)<\cdots\}.$$
Thus for each $y\in B$, we have a partition of $\Z$:
\begin{equation*}
  \Z=\{0\}\cup \{g_{p_1(n)}(y)\}_{n=M}^\infty\cup \{l_{j_i(y)}\}_{i=1}^\infty.
\end{equation*}
Define a permutation $\pi_{p_1,y}: \Z\rightarrow \Z$ by
\begin{equation*}
  \pi_{p_1,y}(i)=\left\{
                 \begin{array}{ll}
                   0, & \hbox{$i=0$;} \\
                   g_{p_1(M+i-1)}(y), & \hbox{$i\ge 1$;} \\
                   l_{j_{-i}(y)}, & \hbox{$i\le -1$.}
                 \end{array}
               \right.
\end{equation*}
Replacing $p_1$ by $p_2$, we can define a permutation $\pi_{p_2,y}$ similarly.
Define $\pi_y: \Z\rightarrow \Z$ as $\pi_y=\pi_{p_1,y}\circ \pi_{p_2,y}^{-1}$. Then
\begin{equation}
  \pi_y(0)=0, \quad \pi_y(g_{p_2(n)}(y))=g_{p_1(n)}(y), \quad  \forall n\in [M,+\infty).
\end{equation}

Given the permutation $\pi_y$ of $\Z$ defined above and $F\subseteq \N$, we define a map $\psi_{\pi_y}: \Sigma\rightarrow \Sigma$ by
\begin{equation*}\label{}
  (\psi_{\pi_y}\w)(q)=\left\{
                     \begin{array}{ll}
                       \w(0), & \hbox{$q=0$;} \\
                       1-\w(\pi_y(q))=1-\w(g_{p_1(n)}(y)), & \hbox{$q=g_{p_2(n)}(y), n\in [M,+\infty)\cap F$;} \\
                       \w(\pi_y(q)), & \hbox{else.}
                     \end{array}
                   \right.
\end{equation*}

Recall that $X=\T\times \Sigma$. Now define $R:\T\times \Sigma \rightarrow \T\times \Sigma$ as follows:
\begin{equation*}
  R(y,\w)=\left\{
            \begin{array}{ll}
              (y, \psi_{\pi_y}\w), & \hbox{$y\in B$;} \\
              (y,\w), & \hbox{$y\in \T\setminus B$.}
            \end{array}
          \right.
\end{equation*}
The required transformation $S: X\rightarrow X$ is then defined by
%We define a transformation $S: X\rightarrow X$ by
$S:= R^{-1}\circ T\circ R.$
\begin{equation*}
\xymatrix
{
\T\times \Sigma \ar[d]_{R}  \ar[r]^{S}  &  \T\times \Sigma\ar[d]^{R} \\
\T\times \Sigma \ar[r]^{T} &  \T\times \Sigma}
\end{equation*}

%\begin{rem}\label{rem-Re}
%In the construction, we introduce the subset $S_0$ to guarantee that %, since we do not know if$\Z\setminus \{f_{p_1(n)}(y)\}_{i=M}^\infty$ is infinite. If we replace $f$ by $g=2f$, then all previous results still hold, and in this case $\Z\setminus \{g_{p_1(n)}(y)\}_{i=M}^\infty$ is infinite, since $g_n$ only takes values in $2\Z$. Thus we can remove $S_0$ from the construction of $(X,\X,\mu,S)$. We thank the referee for pointing out this to us.
%\end{rem}

\subsubsection{}
Note that for $y\in B$ and $n\in \N$, we have $S^n(y,\w)=(R^{-1}\circ T^n\circ R)(y,\w)$. Thus

\begin{small}
\begin{equation}\label{b2}
  S^n(y,\w)=\begin{cases} \big(y+n\a, (\psi^{-1}_{\pi_{y+n\alpha}}\circ \sigma^{g_n(y)}\circ \psi_{\pi_y})(\w)\big), & \text{ if }y+n\alpha\in B;\\
  \big(y+n\a, (\sigma^{g_n(y)}\circ \psi_{\pi_y})(\w)\big), & \text{ if }y+n\alpha \in \mathbb{T}\setminus B.\end{cases}
\end{equation}
\end{small}

\begin{lem}\label{lem-2}
$(X,\X,\mu, S)$ is an ergodic m.p.s. with $h_\mu(X,S)=0$.
\end{lem}

\begin{proof}
To show the result, it suffices to show that $R: X\rightarrow X$ is an invertible measure-preserving transformation according to Proposition \ref{prop-zero}. It is done
by the following steps.

\medskip
\noindent {\bf Step 1}: Write $R$ as the composition of 3 transformations.

\medskip

Given a permutation $\pi$ on $\mathbb{Z}$, we define a map $\phi_\pi: \Sigma\rightarrow \Sigma$ by
\begin{equation*}
  (\phi_\pi \w)(q)=\w(\pi(q)), \quad \forall q\in \Z.
\end{equation*}
Note that $\phi_{\pi^{-1}}=\phi^{-1}_\pi.$
And for $Q\subseteq \N$, we define a map $\phi^Q: \Sigma\rightarrow \Sigma$ by
\begin{equation*}
  (\phi^Q\w)(q)=\left\{
                  \begin{array}{ll}
                    1-\w(q), & \hbox{$q\in Q$;} \\
                    \w(q), & \hbox{$q\not\in Q$.}
                  \end{array}
                \right.
\end{equation*}
Thus for $y\in B$ we have
$$\psi_{\pi_y}=\phi^{Q_y}\circ \phi_{\pi_{p_1,y}}\circ \phi^{-1}_{\pi_{p_2,y}},$$
where $Q_y=\{g_{p_1(n)}(y): n\in [M,+\infty)\cap F\}$.
Define $R_1,R_2,R_3: X\rightarrow X$ by
\begin{equation*}
  R_i(y,\w)=\left\{
            \begin{array}{ll}
              (y, \phi_{\pi_{p_i,y}}\w), & \hbox{$y\in B$;} \\
              (y,\w), & \hbox{$y\in \T\setminus B$}
            \end{array}
          \right.
i=,1,2; \quad R_3(y, \w)=\left\{
            \begin{array}{ll}
              (y, \phi^{Q_y}\w), & \hbox{$y\in B$;} \\
              (y,\w), & \hbox{$y\in \T\setminus B$.}
            \end{array}
          \right.
\end{equation*}
Then
$$R=R_3\circ R_1\circ R_2^{-1}.$$

\noindent {\bf Step 2}: $R_1, R_2$ are isomorphisms.

\medskip
Now we show $R_1$ is a Borel isomorphism. By the definition, it is clear that $R_1: X\rightarrow X$ is a bijective map. By Souslin's Theorem \footnote{Souslin's Theorem says that if $f:X\rightarrow Y$ is a Borel bijection, then $f$ is a Borel isomorphism \cite[Theorem 14.12]{Kechris}.}, it suffices to show that $R_1$ maps Borel measurable subsets to Borel measurable subsets.
Since $R_1|_{\T\setminus B\times \Sigma}={\rm Id}_{\T\setminus B\times \Sigma}$, we need to show that $R_1|_{B\times \Sigma}: B\times \Sigma\rightarrow B\times \Sigma$ maps Borel measurable sets to Borel measurable sets. Let $C\subseteq B$ be a Borel measurable subset and $_{-H}[s]_{H}\in \B(\Sigma)$, where $s=(s_{-H},s_{-H+1},\ldots, s_H)\in \{0,1\}^{2H+1}$, $H\in \N$. We show that $R_1(C\times _{-H}[s]_H)$ is Borel measurable.

Recall that  for $y\in B$
\begin{equation*}
  \Z=\{0\}\cup \{g_{p_1(n)}(y)\}_{n=M}^\infty\cup \{l_{j_i(y)}\}_{i=1}^\infty,
\end{equation*}
where $\{j_1(y)<j_2(y)<\ldots\}=\{j\ge 1: l_j\not\in \{g_{p_1(n)}(y)\}_{n=M}^\infty \}.$
By the definition, for all $y\in C$
\begin{equation}\label{R1-eq-1}
R_1(\{y\}\times _{-H}[s]_H)=\{y\}\times\Big( [s_0]_0\cap \bigcap_{i=1}^H [s_i]_{g_{p_1(M+i-1)}(y)}\cap \bigcap_{i=1}^H[s_{-i}]_{l_{j_i}(y)}\Big).
\end{equation}
Since $g=2f:\T\rightarrow \Z$ is  Borel measurable, for each $J=(m_1,m_2,\ldots, m_H)\in (\Z\setminus\{0\})^H$
the subset $$C^J=\{y\in C: g_{p_1(M)}(y)=m_1,\ldots, g_{p_1(M+H-1)}(y)=m_H\}$$
is Borel measurable. Thus
\begin{equation}\label{c1}
  C=\bigcup_{J\in (\Z\setminus\{0\})^H} C^J, \ \text{and for each $y\in C^J=C^{(m_1,\ldots,m_H)}$}, \ \bigcap_{i=1}^H [s_i]_{g_{p_1(M+i-1)}(y)}=\bigcap_{i=1}^H [s_i]_{m_i}.
\end{equation}

Let $\N^{<H}=\left\{(n_1,n_2,\ldots, n_H)\in \N^H: n_1<n_2<\cdots <n_H\right\}.$ Then $(j_1(y),j_2(y),\ldots, j_H(y))\in \N^{<H}$.
For each $I=(n_1,n_2,\ldots, n_H)\in \N^{<H}$, set
$$C_I=\{y\in C: j_1(y)=n_1,\ldots, j_H(y)=n_H\}.$$
For $q,t\in \N$ set
$$W_{q,t}=\{y\in C: l_q \neq g_{p_1(M+t-1)}(y)\},$$
and it is  Borel measurable.
Then for $y\in C$, $l_q\not\in \{g_{p_1(t)}(y)\}_{t=M}^\infty$ if and only if $y\in \bigcap_{t=1}^\infty W_{q,t}$, and $$l_{n_1},\ldots, l_{n_H}\not\in \{g_{p_1(t)}(y)\}_{t=M}^\infty\ \text{if and only if} \ y\in \bigcap_{r=1}^H\bigcap_{t=1}^\infty W_{n_r,t}.$$
For all $j\in [1,n_H]\setminus \{n_1,n_2,\ldots, n_H\}$, we have $j\in \{g_{p_1(t)}(y)\}_{t=M}^\infty$, which is equivalent to
$$y\in \bigcup_{t=1}^\infty (C\setminus W_{j, t}).$$
Thus
\begin{equation*}
  C_I=\Big(\bigcap_{r=1}^H\bigcap_{t=1}^\infty W_{n_r,t}\Big)\cap \Big(\bigcap_{j\in  [1,n_H]\setminus \{n_1,\ldots, n_H\}}\bigcup_{t=1}^\infty (C\setminus W_{j, t})\Big)
\end{equation*}
is a Borel measurable subset of $C$. And we have
\begin{equation}\label{c2}
  C=\bigcup_{I\in \N^{<H}} C_I, \ \text{and for each $y\in C_I=C_{(n_1,\ldots,n_H)}$}, \ \bigcap_{i=1}^H[s_{-i}]_{l_{j_i}(y)}=\bigcap_{i=1}^H[s_{-i}]_{n_i}.
\end{equation}
By \eqref{c1} and \eqref{c2} we have
\begin{equation*}
  R_1(C\times _{-H}[s]_H)=\bigcup_{J=(m_1,\ldots,m_H)\in (\Z\setminus\{0\})^H\atop I=(n_1,\ldots, n_H)\in \N^{<H}}\big(C^J\cap C_I\big) \times \Big( [s_0]_0\cap \bigcap_{i=1}^H [s_i]_{m_i}\cap \bigcap_{i=1}^H[s_{-i}]_{n_i}\Big).
\end{equation*}
Thus $ R_1(C\times _{-H}[s]_H)$ is Borel measurable. Since the Borel measurable subsets $C\subseteq B$ and $_{-H}[s]_H$ are arbitrary, $R_1^{-1}$ is Borel measurable and by Souslin's Theorem  $R_1: X\rightarrow X$ is a Borel isomorphism (see for example \cite[Theorem 14.12]{Kechris}).

Next we show $R_1,R_1^{-1}$ are measure preserving. Let $C\subseteq B$ be a Borel measurable subset and $_{-H}[s]_{H}\in \B(\Sigma)$, where $s=(s_{-H},s_{-H+1},\ldots, s_H)\in \{0,1\}^{2H+1}$, $H\in \N$. We show that $\mu(R_1(C\times _{-H}[s]_H))=\frac{1}{2^{2H+1}}m(C)=\mu(C\times _{-H}[s]_H)$.

In fact, by \eqref{R1-eq-1} and Fubini's Theorem
\begin{equation*}
 \begin{split}
     & \mu(R_1(C\times _{-H}[s]_H))\\
     =&\mu\left(\bigcup_{y\in C} \{y\}\times\Big( [s_0]_0\cap \bigcap_{i=1}^H [s_i]_{g_{p_1(M+i-1)}(y)}\cap \bigcap_{i=1}^H[s_{-i}]_{l_{j_i}(y)}\Big)\right) \\
   = &  \int_C \nu\Big(  [s_0]_0\cap \bigcap_{i=1}^H [s_i]_{g_{p_1(M+i-1)}(y)}\cap \bigcap_{i=1}^H[s_{-i}]_{l_{j_i}(y)}\Big) dm(y)\\
=& \int_C \frac{1}{2^{2H+1}} dm(y)=\frac{1}{2^{2H+1}}m(C)\\
=&\mu(C\times _{-H}[s]_H).
\end{split}
\end{equation*}
Thus $R_1^{-1}$ is measure-preserving, and hence  $R_1$ is also measure-preserving since $R_1$ is a Borel isomorphism.

To sum up, we have that $R_1: X=\T\times \Sigma\rightarrow \T\times \Sigma$ is an isomorphism. Similarly, we have that $R_2: X=\T\times \Sigma\rightarrow \T\times \Sigma$ is an isomorphism.

\medskip
\noindent {\bf Step 3}. $R_3$ is an isomorphism.

\medskip
Since $R_3|_{\T\setminus B\times \Sigma}={\rm Id}_{\T\setminus B\times \Sigma}$, we need to show that $R_3|_{B\times \Sigma}: B\times \Sigma\rightarrow B\times \Sigma$ maps Borel measurable sets to Borel measurable sets. Let $C\subseteq B$ be a Borel measurable subset and $_{-H}[s]_{H}\in \B(\Sigma)$, where $s=(s_{-H},s_{-H+1},\ldots, s_H)\in \{0,1\}^{2H+1}$, $H\in \N$. We show that $R_3(C\times _{-H}[s]_H)$ is Borel measurable.
By the definition, for all $y\in C$
$$R_1(\{y\}\times _{-H}[s]_H)=\{y\}\times\left( \bigcap_{i\in [-H,H]\cap Q_y} [1-s_i]_{i}\cap \bigcap_{i\in [-H,H]\setminus Q_y}[s_{i}]_{i}\right),$$
where $Q_y=\{g_{p_1(n)}(y): n\in [M,+\infty)\cap F\}$.
For $i\in \mathbb{Z}$, let $$D_i=\{y\in C: i\in Q_y\}=\{y\in C: g_{p_1(n)}(y)=i \ \ \text{for some}\ n\in [M,+\infty)\cap F \},$$
and it is Borel measurable since $g=2f: \T\rightarrow \mathbb{Z}$ is Borel measurable.

 For each $I\subseteq [-H,H]$, let
$$D_I=\bigcap_{i\in I}D_i\cap \bigcap_{i\in [-H,H]\setminus I}(C\setminus D_i).$$
Then we have
\begin{equation*}
  R_1(C\times _{-H}[s]_H)=\bigcup_{I\subseteq [-H,H]}D_I \times \left(\bigcap_{i\in I} [1-s_i]_{i}\cap \bigcap_{i\in [-H,H]\setminus I}[s_{i}]_{i}\right).
\end{equation*}
Thus $ R_3(C\times _{-H}[s]_H)$ is Borel measurable. Since the Borel measurable subsets $C\subseteq B$ and $_{-H}[s]_H$ are arbitrary, $R_3^{-1}$ is Borel measurable and by Souslin's Theorem  $R_3: X\rightarrow X$ is Borel isomorphism. Similar to the proof in {\bf Step 2} for $R_1$, it is easy to verify that $R_3, R_3^{-1}$ are measure-preserving.

Thus we have showed that $R=R_3\circ R_1\circ R_2^{-1}:X\rightarrow X$ is an invertible measure-preserving transformation, and hence $(X,\X,\mu,S)$ is a m.p.s. and it is isomorphic to $(X,\X,\mu,T)$. By Proposition \ref{prop-zero}, $(X,\X,\mu,S)$ is an ergodic m.p.s. with $h_\mu(X,S)=0$. The proof is complete.
\end{proof}

\subsection{Proof of the main theorem}

\subsubsection{} Now we are ready to give the proof of the Main Theorem.
First we need the following lemma. Recall that the subset $B$ and $M\in \N$ are defined in Corollary \ref{cor-1}, and $F$ appears in the construction of $S$.

\begin{lem}\label{lem-3}
Let $A_1=B\times \Sigma$ and $A_2=\T\times [0]_0$. Then for $n\in [M,+\infty)$
\begin{equation}\label{prelastequ}
  \mu(A_1\cap T^{-p_1(n)}A_2\cap S^{-p_2(n)}A_2)=\left\{
                                                   \begin{array}{ll}
                                                     0, & \hbox{$n\in F$;} \\
                                                     \frac{1}{2}m(B), & \hbox{$n\not\in F$.}
                                                   \end{array}
                                                 \right.
\end{equation}
\end{lem}

\begin{proof}
Recall $[0]_j=\{\w\in \Sigma: \w(j)=0\}$. Note that for $n\in [M,+\infty)$, $(y,\w)\in A_1\cap T^{-p_1(n)}A_2\cap S^{-p_2(n)}A_2$ if and only if $y\in B$, $T^{p_1(n)}(y,\w)\in \T\times [0]_0$ and $S^{p_2(n)}(y,\w)\in \T\times [0]_0$.
Moreover, since $(\psi_{\pi_z}^{-1}\widetilde{\w})(0)=\widetilde{\w}(0)$ for any $z\in B$ and $\widetilde{\w}\in \Sigma$, by \eqref{b1} and \eqref{b2}
\begin{equation*}
  \begin{split}
     & A_1\cap T^{-p_1(n)}A_2\cap S^{-p_2(n)}A_2 \\
       = & \{(y,\w)\in X: y\in B, \sigma^{g_{p_1(n)}(y)}\w (0)=0,(\sigma^{g_{p_2(n)}(y)}\circ \psi_{\pi_y})(\w)(0)=0 \}\\
       = & \{(y,\w)\in X: y\in B, \w(g_{p_1(n)}(y))=0,(\sigma^{g_{p_2(n)}(y)}\circ \psi_{\pi_y})(\w)(0)=0 \}.
   \end{split}
\end{equation*}
Note that  $(\sigma^{g_{p_2(n)}(y)}\circ \psi_{\pi_y})(\w)(0)=0$
if and only if $(\psi_{\pi_y}\w)(g_{p_2(n)}(y))=0$. By the definition of $\psi_{\pi_y}$, we have
\begin{equation*}
  (\psi_{\pi_y}\w)(g_{p_2(n)}(y))=\left\{
                                    \begin{array}{ll}
                                      1-\w(g_{p_1(n)}(y)), & \hbox{$n\in [M,\infty)\cap F$;} \\
                                      \w(g_{p_1(n)}(y)), & \hbox{$n\not \in F$.}
                                    \end{array}
                                  \right.
\end{equation*}
It follows that for all $n\in [M,+\infty)$
\begin{equation*}
 A_1\cap T^{-p_1(n)}A_2\cap S^{-p_2(n)}A_2= \left\{
                                              \begin{array}{ll}
                                                \emptyset, & \hbox{$n\in F$;} \\
                                                \bigcup_{y\in B}\Big(\{y\}\times [0]_{g_{p_1(n)}(y)}\Big), & \hbox{$n\not\in F$.}
                                              \end{array}
                                            \right.
\end{equation*}
Note that for $n\in [M,+\infty)\setminus F$,  by Fubuni's Theroem
\begin{equation}\label{lastequ}
\begin{split}
&\mu(A_1\cap T^{-p_1(n)}A_2\cap S^{-p_2(n)}A_2)=\mu\big(\bigcup_{y\in B}\{y\}\times [0]_{g_{p_1(n)}(y)}\big)\\
=&\int_B \nu\big( [0]_{g_{p_1(n)}(y)}\big) dm=\frac{1}{2}m(B).
\end{split}
\end{equation}
The proof is complete.
\end{proof}

\subsubsection{Proof of the Main Theorem}
For any $F\subset\N$ and $c\in (0,\frac12)$, take $\eta=2c\in (0,1)$ in Corollary \ref{cor-1} and 
Lemma \ref{lem-3} gives the first part of the Main Theorem with $c=\frac{1}{2}m(B)$.

Now we show the second part of the Main Theorem. Choose $F$ such that $\lim\limits_{N\to\infty}\frac{|[1,N]\setminus F|}{N}$ does not exist.
By Lemma \ref{lem-3}, we have the limit
\begin{equation*}
  \begin{split}
     & \lim_{N\to\infty}\frac{1}{N} \sum_{n=0}^{N-1} \mu(A_1\cap T^{-p_1(n)}A_2\cap S^{-p_2(n)}A_2) \\
      =&\lim_{N\to\infty}\frac{1}{N} \sum_{n\in [M,N-1]} \mu(A_1\cap T^{-p_1(n)}A_2\cap S^{-p_2(n)}A_2)\\
       =& \lim_{N\to\infty}\frac{\left|[M,N-1]\setminus F\right|}{N}\cdot \frac{1}{2}m(B)
   \end{split}
\end{equation*}
does not exist. In particular, the average
$$\frac{1}{N}\sum_{n=0}^{N-1}1_{A_2}(T^{p_1(n)}x)1_{A_2}(S^{p_2(n)}x)$$
do not converge in $L^2(X,\mu)$ as $N\to\infty$.

%\begin{rem} If $\Z\setminus E$ is finite.\end{rem}

%%%%%%%%%%%%%%%%%%%%%%%%%%%%%%%%%%%%%%%%%%%%%%%%%%%%%%%%

\end{document}